\documentclass[12pt, reqno]{article}
\usepackage{amssymb, mathtools}
\usepackage{latexsym}
\usepackage{amsmath}
\usepackage{amsthm}

\usepackage{epsfig}
\usepackage{verbatim}
\usepackage[all,cmtip]{xy}
\usepackage{enumerate}

\usepackage{authblk}

\usepackage[usenames,dvipsnames]{color}
\usepackage[colorlinks=true,linkcolor=blue,citecolor=blue]{hyperref}

\usepackage{tikz}
\usetikzlibrary{matrix,arrows,decorations.pathmorphing}

\usepackage{float}

\newtheorem*{lemma*}{Lemma}

\newtheorem*{theorem*}{Theorem}

\newtheorem{theorem}{Theorem}[section]
\newtheorem{proposition}[theorem]{Proposition}
\newtheorem{lemma}[theorem]{Lemma}
\newtheorem*{thma}{Theorem A}
\newtheorem*{thmb}{Theorem B}

\newtheorem{corollary}[theorem]{Corollary}

\newtheorem*{conjecture*}{Conjecture}

\theoremstyle{definition}
\newtheorem{definition}[theorem]{Definition}

\newtheorem*{claim*}{Claim}

\theoremstyle{remark}
\newtheorem{remark}[theorem]{Remark}

\theoremstyle{definition}

\newcommand{\diff}{d}

\title{Generalized Euler classes, differential forms and commutative DGAs}
\author[1]{Alexander Gorokhovsky\thanks{Email: \texttt{alexander.gorokhovsky@colorado.edu}; partially supported by NSF grant DMS-0900968.}}
\author[2]{Dennis Sullivan\thanks{Email: \texttt{dennis@math.sunysb.edu};  partially supported by the Einstein Chair Research Fund and CUNY Graduate Center.}}
\author[3]{Zhizhang Xie\thanks{Email: \texttt{xie@math.tamu.edu}; partially supported by NSF grant DMS-1500823.}}
\affil[1]{Department of Mathematics, University of Colorado, Boulder}
\affil[2]{Department of Mathematics, SUNY, Stony Brook}
\affil[3]{Department of Mathematics, Texas A\&M Univeristy}
\date{}

\begin{document}

\maketitle

\begin{abstract}
	In the context of commutative differential graded algebras over $\mathbb Q$, we show that an iteration of ``odd spherical fibration" creates a ``total space" commutative differential graded algebra with only odd degree cohomology. Then we show for such a commutative differential graded algebra that, for any of its ``fibrations"  with ``fiber" of finite cohomological dimension,  the induced map on cohomology  is injective.
\end{abstract}

\section{Introduction}
In geometry, one would like to know which rational cohomology classes in a base space can be annihilated by pulling up to a fibration over the base with finite dimensional fiber. One knows   that if $[x]$ is a $2n$-dimensional rational cohomology class on a finite dimensional CW complex  $X$, there is a $(2n-1)$-sphere fibration over $X$ so that $[x]$ pulls up to zero in the cohomology groups of the total space.  In fact there is a complex vector bundle $V$ over $X$ of rank $n$ whose Euler class is a multiple of $[x]$. Thus this multiple is the obstruction to a nonzero section of $V$, and  vanishes when pulled up to the part of $V$ away from the zero section,  which deformation retracts to the unit sphere bundle.

Rational homotopy theory provides a natural framework to study this type of questions, where topological spaces are replaced by commutative differential graded algebras (commutative DGAs) and topological fibrations replaced by algebraic fibrations. This will be the context in which we work throughout the paper. The reader can read more in  \cite{MR1802847, MR2355774, MR0646078} about the topological meaning of the results of this paper from the perspective of rational homotopy theory of manifolds and general spaces.

The first theorem (Theorem $\ref{thm:itoddsphere}$) of the paper states that the above construction, when iterated, creates a ``total space" commutative DGA with only odd degree cohomology.

\begin{thma}
For each commutative DGA  $(A, d)$, there exists an iterated odd algebraic spherical fibration $(TA, d) $ over $(A, d)$ so that  all even cohomology [except dimension zero] vanishes.
\end{thma}

Our next theorem (Theorem $\ref{thm:oddsphere}$) then limits the odd degree classes that can be annihilated by fibrations whose fiber has finite cohomological dimension.

\begin{thmb}
Let $(B, \diff)$ be a connected commutative DGA  such that $H^{2k}(B)=0$ for all $0 < 2k  \leq 2N$. If $\iota \colon (B, \diff) \to (B\otimes \Lambda V, \diff) $ is an algebraic fibration whose  algebraic fiber has finite cohomological dimension, then the induced map
\[ \iota_\ast \colon \bigoplus_{i\leq 2N} H^i(B)\to  \bigoplus_{i\leq 2N} H^i(B\otimes \Lambda V)\]
is injective.
\end{thmb}

It follows from the two theorems above that the iterated odd spherical fibration construction is universal for cohomology classes that pull back to zero by any fibrations whose fiber has finite cohomological dimension.

The paper is organized as follows. In Section $\ref{sec:pre}$, we recall some definitions from rational homotopy theory. In Section $\ref{sec:gysin}$, we use iterated algebraic spherical fibrations to prove Theorem A. In Section  $\ref{sec:algsphere}$, we define bouquets of algebraic spheres and analyze their minimal models. In Section $\ref{sec:main}$, we prove Theorem B.

The work in this paper started during a visit of the authors at IHES. The authors would like to thank IHES for its hospitality and for providing a great environment.

\section{Preliminaries} \label{sec:pre}

We recall some definitions related to commutative differential graded algebras. For more details, see \cite{MR1802847, MR2355774, MR0646078}.

\begin{definition}
	A commutative differential graded algebra (commutative DGA) is a graded algebra $B = \oplus_{i\geq 0} B^{i}$ over $\mathbb Q$ together with a differential $\diff\colon B^i \to B^{i+1}$ such that $\diff^2 = 0$,  $xy = (-1)^{ij} yx,$ and  $\diff(xy) = (\diff x) y + (-1)^{i} x (\diff y)$,  for all $x\in B^{i}$ and  $y\in B^{j}$.
	
\end{definition}

\begin{definition}
	\begin{enumerate}[(1)]
		\item
		A commutative DGA $(B, \diff)$ is called connected if $B^0 = \mathbb Q$.
		\item A commutative DGA $(B, \diff) $ is called simply connected if  $(B, \diff)$ is connected and $H^1(B) = 0$.
		\item
		A commutative DGA $(B, \diff)$ is of finite type if $H^k(B)$ is finite dimensional for all $k\geq 0$.
		\item A commutative DGA $(B,\diff )$ has finite cohomological dimension $d$, if $d$ is the smallest integer such that $H^k(B) = 0 $ for all $k> d$.
	\end{enumerate}
\end{definition}

\begin{definition}
	A connected commutative DGA $(B, \diff)$ is called a model algebra if as a commutative graded algebra it is free on a set of generators $\{x_\alpha\}_{\alpha\in \Lambda}$ in positive degrees, and these generators can be partially ordered so that $\diff x_\alpha$ is an element in the algebra generated by $x_\beta$ with $\beta < \alpha$.

\end{definition}

\begin{definition}
	A model algebra $(B, d)$ is called minimal if for each generator $x_\alpha$,  $\diff x_\alpha$ has no linear term, that is,
	\[ \diff(B) \subset B^+ \cdot B^+, 	\textup{ where } B^+ = \oplus_{k>0} B^k.\]

\end{definition}

\begin{remark}
	For every connected commutative DGA $(A, \diff_A)$, there exists a  minimal model algebra $(\mathcal M(A), \diff)$ and a morphism $\varphi\colon (\mathcal M(A), \diff)\to (A, \diff_A)$ such that $\varphi$ induces an isomorphism on cohomology. $(\mathcal M(A), \diff)$ is called a minimal model of $(A, \diff)$, and is unique up to isomorphism. See page $288$ of \cite{MR0646078} for more details, cf. \cite{MR1802847, MR2355774}.
\end{remark}

\begin{definition}\label{def:algfib}
	\begin{enumerate}[(i)]
		\item An algebraic fibration (also called \textit{relative model algebra}) is an inclusion of commutative DGAs $(B, \diff) \hookrightarrow (B\otimes \Lambda V, \diff)$ with $V = \oplus_{k\geq 1} V^k$ a graded vector space; moreover,  $V = \bigcup_{n=0}V(n)$, where $V(0) \subseteq V(1) \subseteq V(2) \subseteq \cdots $ is an increasing sequence of graded subspaces of $V$ such that
		\[ d : V(0) \to B \quad \textup{ and } \quad  d: V(n) \to B\otimes \Lambda V(n-1), \quad n\geq 1,\]
		where $\Lambda V$ is the free commutative DGA generated by $V$.
		\item An algebraic fibration is called minimal  if
		\[  \textup{Im}(\diff) \subset B^+\otimes \Lambda V + B\otimes \Lambda^{\geq 2} V.\]
		
	\end{enumerate}

\end{definition}

Let $\iota\colon (B, \diff) \hookrightarrow (B\otimes \Lambda V, \diff)$ be an algebraic fibration. Suppose $B$ is connected.  Consider the canonical augmentation morphism $\varepsilon: (B, \diff) \to (\mathbb Q, 0)$ defined by  $\varepsilon (B^+) = 0$. It  naturally induces a commutative DGA:
\[ (\Lambda V, \bar \diff) := \mathbb Q\otimes_B (B\otimes \Lambda V, \diff).\]
We call $(\Lambda V, \bar \diff) $ the algebraic fiber of the given algebraic fibration.

\section{Iterated odd spherical algebraic fibrations}\label{sec:gysin}

In this section, we show that for each commutative DGA, there exists an iterated odd algebraic spherical fibration over it such that the total commutative DGA has only odd degree cohomology.

Let $(B, \diff)$ be a connected commutative DGA. An \emph{odd algebraic spherical fibration} over $(B, \diff)$  is an inclusion of commutative DGAs of the form
\[ \varphi: (B, \diff) \to (B\otimes \Lambda (x), \diff), \]
such that $\diff x \in B$,  where  $x$ has degree $2k-1$ and $\Lambda(x)$ is the free commutative graded algebra generated by $x$. The element $e = \diff x \in B$ is called the Euler class of this algebraic spherical fibration.

\begin{proposition}\label{prop:gysin}
 Let $(B, \diff)$ be a commutative DGA. For every even dimensional class $\beta\in H^{2k}(B)$ with $k>0$, there exists an odd algebraic spherical fibration $\varphi\colon (B, d) \to (B\otimes \Lambda(x), \diff)$ such that its Euler class is equal to $\beta$ and the kernel of the map $\varphi_\ast\colon H^{i+2k}(B) \to H^{i+2k}(B\otimes \Lambda (x) )$ is  $H^i(B)\cdot\beta = \{a\cdot \beta \mid a\in H^i(B)\}$.
 \end{proposition}
\begin{proof}
	Let $(B\otimes \Lambda (x), \diff)$ be the commutative DGA obtained from $(B, \diff)$ by adding a generator $x$ of degree $2k-1$ and defining its differential to be $\diff x = \beta$. We have the following short exact sequence
		\[  0 \to (B, \diff) \to (B\otimes \Lambda (x), \diff) \to ( B \otimes (\mathbb Q\cdot x), \diff\otimes \textup{Id}) \to 0,\]
		which induces a long exact sequence 	 \[ \cdots \to H^{i-1}(B\otimes (\mathbb Q \cdot x)) \to H^i(B) \to H^i(B\otimes \Lambda (x)) \to H^i(B\otimes (\mathbb Q \cdot x)) \to \cdots.
		\]
	Applying the identification  $H^{i + (2k-1)}(B\otimes (\mathbb Q \cdot x)) \cong  H^i(B)$, we obtain the following Gysin sequence
	\[ \cdots \to H^{i}(B) \xrightarrow{\cup e} H^{i+2k}(B) \xrightarrow{\varphi_\ast} H^{i+2k}(B\otimes \Lambda (x)) \xrightarrow{\partial_{i+1}} H^{i+1}(B) \to \cdots.  \]
	This finishes the proof.
\end{proof}

\begin{definition}
	An \emph{iterated odd algebraic spherical  fibration} over $(B, d)$ is algebraic fibration $(B, \diff) \hookrightarrow (B\otimes \Lambda V, \diff)$
	such that  $V^k=0$ for $k$ even.     This fibration is called  \emph{finitely iterated odd algebraic spherical fibration} if $\dim V < \infty$.
	
\end{definition}

Now let us prove the main result of this section.

\begin{theorem}\label{thm:itoddsphere}
	For each commutative DGA  $(A, d)$, there exists an iterated odd algebraic spherical fibration $(TA, d) $ over $(A, d)$ such that   all even cohomology [except dimension zero] vanishes.
\end{theorem}

\begin{proof}
	We will construct $TA$ by induction. In the following, for notational simplicity,  we shall omit the differential $\diff$ from our notation.
	
	Let $\mathcal A_0 = A$. Suppose we have defined the iterated odd algebraic spherical fibration $\mathcal A_{m-1}$ over $A$. Fix a basis of $H^{2k}(\mathcal A_{m-1})$ for each $k>0$. Denote the union of all these bases by  $\{a_i\}_{i\in I}$. Define $W_{m-1}$ to be a $\mathbb Q$ vector space with basis $\{x_i\}_{i\in I}$, where $|x_i|= |a_i| -1$.
	 We define $\mathcal A_{m}$ to be the iterated odd algebraic spherical fibration $ \mathcal A_{m-1} \otimes \Lambda (W_{m-1})$ over $\mathcal A_{m-1}$ with $\diff x_i = a_i$ for all $i\in I$. The inclusion map $\iota\colon \mathcal A_{m-1} \hookrightarrow \mathcal A_{m}$ induces the zero map $\iota_\ast = 0 \colon H^{2k}(\mathcal A_{m-1}) \to H^{2k}(\mathcal A_{m})$ for all $k>0$.
By construction,   	$\mathcal A_m$ is also an iterated odd algebraic spherical fibration.
	
		Finally, we define $TA$ to be the direct limit of $\mathcal A_m$ under the inclusions $\mathcal A_m \hookrightarrow \mathcal A_{m+1}$. Clearly,  $TA$ is an iterated odd algebraic spherical fibration over $A$. More precisely, let $V = \bigcup_{i=0}^\infty W_{i}$. We have $TA = A\otimes \Lambda V$ with the filtration of $V$ given by $V(n) = \bigcup_{i=0}^n W_{i}$.  Moreover, we have $H^{2k}(TA) = 0$ for all $2k>0$. This completes the proof.

\end{proof}

\begin{remark}\label{rm:finsphere}
	If an element $\alpha \in H^\bullet(A) $ maps to zero in $H^{\bullet}(TA)$, then there exists  a subalgebra $S_\alpha$ of $TA$ such that $S_{\alpha}$ is a \emph{finitely} iterated odd algebraic spherical  fibration over $A$ and  $\alpha$ maps to zero in $H^\bullet(S_{\alpha})$.
\end{remark}

\section{Bouquets of algebraic spheres}
 \label{sec:algsphere}
In this section, we introduce a notion of bouquets of algebraic  spheres. It is an algebraic analogue of usual bouquets of spheres in topology.

\begin{definition}\label{def:algsphere}
	For a given set of generators $X= \{x_i\}$ with $x_i$ having odd degree $|x_i|$, we define the bouquet of odd algebraic spheres labeled by $X$ to be the following commutative DGA
	\[\mathcal S(X) = \left( \bigwedge_{x_i\in X} \mathbb Q[x_i]\right)\Big/\langle x_ix_j = 0 \mid \textup{all } i, j \rangle  \]
	with the differential $d = 0$.
\end{definition}

\begin{proposition}\label{prop:algsphere}
	Let $\mathcal S(X)$ be a bouquet of odd algebraic spheres, and $\mathcal M(X) = (\Lambda V, \diff)$ be its minimal model. Then $\mathcal M(X)$ satisfies the following properties:
	\begin{enumerate}[(i)]
		\item $\mathcal M$ has no even degree generators, that is, $V$ does not contain even degree elements;
		\item each element in $H^{\geq 1}(\mathcal M(X))$ is represented by a generator, that is, an element in $V$.
	\end{enumerate}
\end{proposition}
\begin{proof}
	This is a special case of Koszul duality theory, cf. \cite[Chapter 3, 7 \& 13]{MR2954392}. Since $\mathcal S = \mathcal S(X)$ has zero differential, we may forget its differential and view it as a  graded commutative algebra.   An explicit construction of a minimal model of $\mathcal S$ is given as follows:  first take the Koszul dual coalgebra $\mathcal S^{\text{\textexclamdown}}$ of $\mathcal S$; then apply the cobar construction to $\mathcal S^{\text{\textexclamdown}}$, and  denote the resulting commutative DGA by $\Omega \mathcal S^{\text{\textexclamdown}}$. By Koszul duality, $\mathcal M(X) \coloneqq \Omega \mathcal S^{\text{\textexclamdown}}$ is a minimal model of $\mathcal S$.
	
More precisely, set $W = \bigoplus_{i\geq 0} W_{i}$ to be the graded vector space spanned by $X$. Let $sW$ (resp. $s^{-1}W$) be the suspension (resp. desuspension) of $W$, that is, $(sW)_{i-1} = W_i$ (resp. $(s^{-1}W)_{i} = W_{i-1}$).   Let  $\mathcal L^c = \mathcal L^c(sW)$ be the cofree Lie coalgebra generated by $sW$. More explicitly, let $T^c(sW) = \bigoplus_{n\geq 0}(sW)^{\otimes n}$ be the tensor coalgebra, and $T^c(sW)^+ = \bigoplus_{n\geq 1} (sW)^{\otimes n}$. The coproduct on $T^c(sW)$ naturally induces a Lie cobraket on $T^c(sW)$. Then we have $\mathcal L^c(sW) = T^c(sW)^+/T^c(sW)^+\ast T^c(sW)^+$, where $\ast$ denotes the shuffle multiplication.  With the above notation, we have  $\mathcal S^{\text{\textexclamdown}} \cong \mathcal L^c$.  The cobar construction of $\mathcal L^c$ is given explicitly by
\[  \mathbb Q\to s^{-1}\mathcal L^c \xrightarrow{\ d \ } \Lambda^2 (s^{-1}\mathcal L^c) \to \cdots \to \Lambda^n(s^{-1}\mathcal L^c) \to \cdots \]
with the differential $d$ determined by the Lie cobraket of $\mathcal L^c$.  Now the desired properties of $\mathcal M(X)$  follow from this explicit construction.
\end{proof}

\begin{remark}
	In the special case of a bouquet of odd  algebraic spheres where the cohomology of a commutative DGA model is that of a circle or the first Betti number  is  zero,  this was discussed  by Baues \cite[Corollary 1.2]{MR0442922} and by Halperin and Stasheff \cite[Theorem 1.5]{MR539532}.
	
\end{remark}

\section{Main theorem} \label{sec:main}

In this section, we show that if a commutative DGA has cohomology,  up to a  certain degree, isomorphic to that of a bouquet of odd algebraic spheres, then its minimal model is isomorphic to that of the bouquet of odd algebraic spheres, up to that given degree. Then we apply it to prove that if a commutative DGA has only odd degree cohomology up to a certain degree, then all nonzero cohomology classes
up to that degree will never pull back to zero by any algebraic fibration whose fiber has finite cohomological dimension.

Suppose $B$ is a connected  commutative DGA of finite type such that $H^{2k}(B)=0$ for all $ 0 < 2k \le 2N$. Let  $X_i$ be  a basis of $H^{i}(B)$ and $X = \bigcup_{i=1}^{2N+1}X_i$. Let  $M = \mathcal M(X)$ be the bouquet of odd algebraic spheres labeled by $X$ from Definition $\ref{def:algsphere}$. Then we have $H^{i}(M) \cong H^i(B)$ for all $0\leq i \leq 2N$. Let $M_k \subset M$ be the subalgebra generated
by the generators of degree $\leq k$.

\begin{lemma}Let  $k$ be an odd integer. Then $H^{k+2}(M_k)=H^{k+1}(M_k) =0$.
\end{lemma}
\begin{proof}
	$H^{k+1}(M_k)=0$ as $H^{k+1}(M_k) \to  H^{k+1}(M)=0$ is injective.
	
	By  Proposition $\ref{prop:algsphere}$ above,  $M$ has no even-degree generators. In particular, we have $M_k=M_{k+1}$. Moreover,
	$H^{\geq 1}(M)$ is spanned by odd-degree generators. From the first observation  it
	follows that the map $H^{k+2}(M_k) \to  H^{k+2}(M)$ is injective, and from the second that its range is $0$.
	
\end{proof}

It follows that for an odd $k$, we have  $M_{k+2} =M_k \otimes \Lambda (V[k+2])$ as an algebra,  where the vector space $V = V_1 \oplus V_2 $ is placed at degree $(k+2)$,  with $V_1 \cong  H^{k+2}(M)$ and $V_2 = H^{k+3}(M_{k})$.
The  differential can be described as follows. It suffices to define $\diff \colon V \to  M_{k}$. We define $\diff=0$ on $V_1$.
To define $\diff$ on $V_2$, let us choose a basis $\{a_i\}_{i\in I}$ of $H^{k+3}(M_k)$. Let $\{\tilde a_i\}_{i\in I}$ be the corresponding basis of $V_2$. Then we define $\diff \tilde a_i =a_i$.

\begin{proposition} \label{prop:constr}
	For each odd integer $k\le 2N$, there exists a morphism $\varphi_k \colon M_k \to B$ such that
	the induced map on cohomology $H^i(M_k)\cong H^i(M) \to H^i(B)$ is an isomorphism for $i \le k$.
\end{proposition}
\begin{proof}We construct the maps $\varphi_k$ by induction. By the previous lemma and the fact that $M$ has no even degree generators, it suffices to define $\varphi_k$ for odd integers $k$ . The case where $k=1$ is clear.
	
Now assume that we have constructed $\varphi_n$, with $n$ an odd integer $\leq 2N-3$. We shall extend $\varphi_n$ to a morphism $\varphi_{n+2}$ on $M_{n+2} = M_n\otimes \Lambda (V[n+2])$,  where the vector space $V = V_1 \oplus V_2 $ is placed at degree $(n+2)$,  with $V_1 \cong  H^{n+2}(M)$ and $V_2 = H^{n+3}(M_{n})$. It suffices to define $\varphi_{n+2}$ on $V$. Let $\{b_j\}_{j\in J}$ be a basis of $H^{n+2}(B)$. Since $H^{n+2}(M)\cong H^{n+2}(B)$, let  $\{\tilde b_j\}_{j\in J}$ be the corresponding basis of $V_1$. We define $\varphi_{n+2}$ on $V_1$ by setting $\varphi_{n+2}(\tilde b_j) = b_j$. Similarly, choose a basis $\{c_\lambda\}_{\lambda \in K}$  of $H^{n+3}(M_n)$, and let $\{\tilde c_\lambda\}_{\lambda \in K}$ be the corresponding
	basis of $V_2$. Since $H^{n+3}(B) = 0$, for each $c_\lambda \in M_n$, there exists $\theta_\lambda\in B$ such that $\varphi_n(c_\lambda) = \diff \theta_\lambda$. We define $\varphi_{n+2}$ on $V_2$ by setting $\varphi_{n+2}(\tilde c_\lambda) =\theta_\lambda$.
	By construction, the induced map $(\varphi_{n+2})_\ast$ on $H^i$ agrees with $(\varphi_n)_\ast$ for $i\leq n+1$ and $(\varphi_{n+2})_\ast$ is an isomorphism on $H^{2n+2}$. This finishes the proof.

\end{proof}

Now let $\mathcal{M}_B$ be a minimal model of $B$ and $(\mathcal{M}_B)_k$ be the subalgebra generated by the generators of degree $\le k$. Combining the above results, we have proved the following proposition.

\begin{proposition} \label{prop:sw} The commutative DGAs $(\mathcal{M}_B)_{2N-1}$ and $M_{2N-1}$ are isomorphic.
\end{proposition}

Moreover, we have the following result, which is an immediate consequence of the construction in Proposition $\ref{prop:constr}$.
\begin{corollary}\label{cor:tosphere}
	Let $B$ be a connected commutative DGA such that $H^{2i}(B)=0$ for all  $ 0 < 2i  \le 2N$. Let $\alpha$ be a nonzero class in $H^{2k+1}(\mathcal{M}_B)$ with $2k+1<2N$.
	Then there exists a morphism $\psi \colon \mathcal{M}_B \to (\Lambda(\eta), 0)$  such that $\psi_*(\alpha)=[\eta]$, where $\eta$ has degree $2k+1$ and $\Lambda(\eta)$ is the free commutative graded algebra generated by $\eta$.
\end{corollary}
\begin{proof} From the description of the minimal model $\mathcal M_B$ of $B$,  it follows that $\mathcal M_B$ has a set of generators such that
	all the cohomology groups up to degree $(2N-1)$ is generated by the cohomology classes of these generators; moreover we can choose these generators
	so that the given class $\alpha$ is represented by a generator, say,  $a$. Then we define $\psi$ by mapping $a$ to $\eta$ and the other generators to $0$.
	
\end{proof}

An inductive application of the same argument above proves the following.

\begin{proposition}
	Suppose $(C, d)$ is a connected commutative DGA with $H^{2k}(C) = 0 $ for all $2k>0$.  Let $X_i$ be  a basis of $H^i(C)$ and $X_C = \bigcup_{i=1}^\infty X_i$. Then the bouquet of odd algebraic spheres $\mathcal M(X_C)$ is a minimal model of $(C, d)$.
\end{proposition}

Applying the above proposition to the commutative DGA $(TA, \diff)$ from Theorem $\ref{thm:itoddsphere}$ immediately gives us the following corollary.
\begin{corollary}\label{cor:ta}
	With the same notation as above, the minimal model of $(TA, d)$ is isomorphic to a bouquet of odd algebraic spheres.
\end{corollary}

Before proving the main theorem of this section, we shall prove the following special case first.

\begin{theorem}\label{thm:oddsphere}
Let  $(\Lambda (x), \diff)$ be the commutative DGA generated by  $x$ of degree $2k+1\geq 1$ such that  $\diff x=0$. For any algebraic fibration $\varphi: (\Lambda(x), \diff) \to (\Lambda(x)\otimes \Lambda V, \diff)$ whose algebraic fiber $(\Lambda V, \bar \diff )$  has finite cohomological dimension, the map $\varphi_\ast: H^j(\Lambda (x)) \to H^j(\Lambda(x)\otimes \Lambda V)$ is injective for all $j$.
	
\end{theorem}
\begin{proof}
	The case where $2k+1 = 1$ is trivial. Let us assume $2k+1>1$ in the rest of the proof.
	
	Let $\varphi\colon (\Lambda V, d) \hookrightarrow (\Lambda(x)\otimes\Lambda V, d)$ be any algebraic fibration whose algebraic fiber has finite cohomological dimension. It suffices to show that $\varphi_\ast\colon H^{2k+1}(\Lambda(x)) \to H^{2k+1}(\Lambda(x)\otimes \Lambda V)$ is injective, since the induced map  $\varphi_\ast$ on $H^i$ is automatically injective for $i\neq 2k+1$.
	
	Now suppose to the contrary that
	\[\varphi_\ast(x) = 0 \textup{ in } H^{2k+1}(\Lambda(x)\otimes \Lambda V). \]
	Then we have $ x = \diff (w\cdot x + v)$
	for some $w, v\in \Lambda V$. By inspecting the degrees of the two sides, one sees that $w = 0$. Therefore, we have  $x = \diff v$ for some $v\in \Lambda V$. It follows that $\bar \diff v = 0$.
	
	Now let $n\in \mathbb N$ be the smallest integer such that $[v^n] = 0$ in $H^{\bullet}(\Lambda V, \bar\diff)$. Such an integer exists since $(\Lambda V, \bar\diff)$ has finite cohomological dimension. Then there exists $u\in \Lambda V$ such that $v^n = \bar \diff u$. Equivalently, we have
	\[  v^n =  u_0\cdot x + \diff u,\]
	for some $u_0\in \Lambda V$. It follows that
	\[ 0 = \diff^2 u = \diff( v^n - u_0\cdot x) = nv^{n-1}\cdot x - (du_0)\cdot x. \]
	Therefore, $v^{n-1} = \frac{1}{n} \diff u_0$, which implies that $[v^{n-1}] = 0$ in $H^{\bullet}(\Lambda V, \bar\diff)$. We arrive at a contradiction. This completes the proof.

\end{proof}

Now let us prove the main result of this section.

\begin{theorem} \label{thm:inj} Let $(B, \diff)$ be a connected commutative DGA  such that $H^{2k}(B)=0$ for all $0 < 2k  \leq 2N$. If $\iota \colon (B, \diff) \to (B\otimes \Lambda V, \diff) $ is an algebraic fibration whose  algebraic fiber has finite cohomological dimension, then the induced map
	\[ \iota_\ast \colon \bigoplus_{i< 2N} H^i(B)\to  \bigoplus_{i< 2N} H^i(B\otimes \Lambda V)\]
	is injective.
\end{theorem}
\begin{proof}
	
Let $f\colon (\mathcal M_B, \diff)\to (B, \diff)$ be a minimal model algebra of $B$.
\begin{claim*}
	For any algebraic fibration $\iota \colon (B, \diff) \to (B\otimes \Lambda V, \diff) $, there exist an algebraic fibration $\varphi\colon (\mathcal M_B, \diff) \to (\mathcal M_B\otimes \Lambda V, \diff)$ and a quasi-isomorphism $g\colon (\mathcal M_B\otimes \Lambda V, \diff) \to (B\otimes \Lambda V, \diff)$ such that the  following diagram commutes:
	\[  \xymatrix{ \mathcal M_B  \ar@{^{(}->}[d]_{\varphi} \ar[r]^f & B \ar@{^{(}->}[d]^\iota \\
		\mathcal M_B\otimes \Lambda V \ar[r]^-{g} &  B\otimes \Lambda V.} \]
\end{claim*}
We construct $\varphi$ and $g$ inductively. Consider the filtration $V = \cup_{n=0}^\infty V(k)$ from Definition $\ref{def:algfib}$. Choose a basis $\{x_i\}_{i\in I_0}$ of $V(0)$. Let $x = x_i$ be a basis element. If $\diff x = a \in B$, then $\diff a = \diff^2 x = 0$. It follows that there exists $\tilde a \in \mathcal M_B$ such that $f(\tilde a) = a + \diff c$ for some $c\in B$. We define an algebraic fibration $\varphi_0\colon (\mathcal M_B, d) \hookrightarrow (\mathcal M_B\otimes \Lambda(x), \diff)$ by setting $\diff x = \tilde a$. Moreover, we extend $f\colon (\mathcal M_B, d)  \to (B, d)$ to a morphism (of commutative DGAs) $g_0\colon (\mathcal M_B\otimes \Lambda (x), \diff) \to (B\otimes \Lambda (x), \diff)$ by setting $g(x) = x + c$. By the Gysin sequence from Section $\ref{prop:gysin}$, we see that $g_0$ is a quasi-isomorphism. Now apply the same construction to all basis elements $\{x_i\}_{i\in I_0}$. We still denote the resulting morphisms by $\varphi_0 \colon (\mathcal M_B, \diff) \to (\mathcal M_B\otimes \Lambda (V(0)), \diff)$ and $g_0\colon (\mathcal M_B\otimes \Lambda (V(0)), \diff) \to (B\otimes \Lambda (V(0)), \diff)$.

Now suppose we have constructed an algebraic fibration
\[ \varphi_{k} \colon (\mathcal M_B\otimes \Lambda (V(k-1)), \diff) \to (\mathcal M_B\otimes \Lambda (V(k)), \diff)\] and a quasi-isomorphism $g_k\colon (\mathcal M_B\otimes \Lambda (V(k)), \diff) \to (B\otimes \Lambda (V(k)), \diff)$ such that
the  following diagram commutes:
\[  \xymatrixcolsep{4pc}\xymatrix{ \mathcal M_B\otimes \Lambda (V(k-1))  \ar@{^{(}->}[d]_{\varphi_k} \ar[r]^{g_{k-1}} & B\otimes \Lambda(V(k-1)) \ar@{^{(}->}[d]^\iota \\
	\mathcal M_B\otimes \Lambda (V(k)) \ar[r]^-{g_k} &  B\otimes \Lambda (V(k)).} \]
Let $\{y_i\}_{i\in I_{k+1}}$ be a basis of $V(k+1)$ that extends the basis $\{x_i\}_{i\in I_{k}}$ of $V(k)\subseteq V(k+1)$. Apply the same construction above to elements in $\{y_i\}_{i\in I_{k+1}} \backslash \{x_i\}_{i\in I_{k}}$, but with $B\otimes \Lambda (V(k))$ in place of  $B$,  and  $\mathcal M_B\otimes \Lambda (V(k))$
in place of $\mathcal M_B$.

We define $(\mathcal M_B\otimes \Lambda V, d)$ to be the direct limit of  $(\mathcal M_B\otimes \Lambda (V(k)), \diff)$ with respect to the morphisms $\varphi_{k} \colon (\mathcal M_B\otimes \Lambda (V(k-1)), \diff)$. We define $\varphi$ to be the  natural inclusion morphism  $ (\mathcal M_B, \diff) \hookrightarrow (\mathcal M_B\otimes \Lambda V, \diff)$. The morphisms $g_k$ together also induce a quasi-isomorphism $g\colon (\mathcal M_B\otimes \Lambda V, \diff) \to (B\otimes \Lambda V, \diff)$, which makes the diagram in the claim  commutative. This finishes the proof of the claim.

Now assume to the contrary that there exists $0\neq \alpha\in H^{2k+1}(B)$ with $2k+1 < 2N$ such that $\iota_\ast(\alpha) = 0$. Let $\tilde \alpha \in H^{2k+1}(\mathcal M_B)$ be the class such that $f_\ast(\tilde \alpha) = \alpha$. In particular, we have $\varphi_{\ast}(\tilde \alpha) = 0$.  By Corollary $\ref{cor:tosphere}$, there exists a morphism $\psi\colon  (\mathcal M_B, \diff) \to (\Lambda (\eta), 0)$  such that $\psi_{\ast}(\tilde \alpha) = \eta$. Now let \[ \tau\colon (\Lambda(\eta), 0) \to (\Lambda(\eta) \otimes \Lambda V, \diff) = (\Lambda(\eta) \otimes_{\mathcal M_B} (\mathcal M_B\otimes \Lambda V), \diff)\] be the push-forward algebraic fibration of $\varphi\colon (\mathcal M_B, \diff) \to (\mathcal M_B\otimes \Lambda V, \diff)$. It follows that
\[ \tau_\ast(\eta) = \tau_\ast \psi_{\ast}(\tilde \alpha) = (\psi\otimes 1)_\ast\varphi_{\ast}(\tilde \alpha) = 0 \]
which contradicts Theorem $\ref{thm:oddsphere}$. This completes the proof.

\end{proof}

\end{document}